\documentclass[11pt,leqno,amscd,amssymb,pstricks,verbatim]{amsproc}
\usepackage{amsfonts}
\usepackage{amssymb}
\include{micro}

\newtheorem{thm}{Theorem}[section]
\newtheorem{lem}[thm]{Lemma}

\theoremstyle{definition}

 \theoremstyle{remark}
\newtheorem{rem}[thm]{Remark}
\newtheorem{cor}[thm]{Corollary}

\newtheorem{ques}[thm]{Question}
\newtheorem{claim}[thm]{Claim}

\numberwithin{equation}{section}

\begin{document}

\def\frakl{{\mathfrak L}}
\def\frakg{{\mathfrak G}}
\def\bbf{{\mathbb F}}
\def\bbl{{\mathbb L}}
\def\bbz{{\mathbb Z}}
\def\bbr{{\mathbb R}}

\def\bvp{\bf{\varphi}}

\def\Der{\mbox{\rm Der}}
\def\Hom{\text{\rm Hom}}
\def\Ker{\text{\rm Ker}}
\def\Lie{\text{\rm Lie}}
\def\id{\mbox {\bf id}}
\def\det{\mbox{\rm det}}
\def\Lie{\mbox {\rm Lie}}
\def\Hom{\mbox {\rm Hom}}
\def\Aut{\mbox {{\rm Aut}}}
\def\Ext{\mbox {{\rm Ext}}}
\def\Coker{\mbox {{\rm Coker}}}
\def\dim{\mbox{{\rm dim}}}
\def\pr{\mbox{pr}}

\def\geqs{\geqslant}

\def\ba{{\mathbf a}}
\def\bd{{\mathbf d}}
\def\co{{\mathcal O}}
\def\cn{{\mathcal N}}
\def\cv{{\mathcal V}}
\def\cz{{\mathcal Z}}
\def\cq{{\mathcal Q}}
\def\cf{{\mathcal F}}
\def\cc{{\mathcal C}}
\def\ca{{\mathcal A}}

\def\ggg{{\frak g}}
\def\lll{{\frak l}}
\def\hhh{{\frak h}}
\def\nnn{{\frak n}}
\def\sss{{\frak s}}
\def\bbb{{\frak b}}
\def\ccc{{\frak c}}
\def\ooo{{\mathfrak o}}
\def\ppp{{\mathfrak p}}
\def\uuu{{\mathfrak u}}

\def\p{{[p]}}
\def\modf{\text{{\bf mod}$^F$-}}
\def\modr{\text{{\bf mod}$^r$-}}

\title[Irreducibility of
Parabolic Baby Verma Modules] {A Criterion for Irreducibility of
Parabolic Baby Verma Modules of Reductive Lie Algebras}
\author{Yi-Yang Li, Bin Shu and Yu-Feng Yao}
\address{School of Fundamental Studies, Shanghai University of Engineering Science,
Shanghai 201620, China.}\email{yiyang$\_$li1979@aliyun.com}
\address{Department of Mathematics, East China Normal University,
Shanghai 200241,  China.} \email{bshu@math.ecnu.edu.cn}
\address{Department of Mathematics, Shanghai Maritime University, Shanghai 201306,
China.}\email{yfyao@shmtu.edu.cn} 
\thanks{This work is supported by the National Natural Science Foundation of China (Grant Nos. 11201293 and 11271130) and the Innovation Program of Shanghai Municipal Education Commission (Grant Nos. 12ZZ038 and 13YZ077).}

\begin{abstract} Let $G$ be a connected, reductive algebraic group over
an algebraically closed field $k$ of prime characteristic $p$ and
$\ggg=\Lie(G)$. In this paper, we study representations of $\ggg$
with a $p$-character $\chi$ of standard Levi form. When
$\ggg$ is of type $A_n, B_n, C_n$ or $D_n$, a sufficient condition
for the irreducibility of standard parabolic baby Verma
$\ggg$-modules is obtained. This partially answers a question raised
by Friedlander and Parshall in [Friedlander E. M. and Parshall B.
J., {\em Deformations of Lie algebra representations}, Amer. J.
Math. 112 (1990), 375-395]. Moreover, as an application, in the
special case that $\ggg$ is of type $A_n$ or $B_n$, and $\chi$ lies
in the sub-regular nilpotent orbit, we recover a result of Jantzen
in [Jantzen J. C., {\em Subregular nilpotent representations of
$sl_n$ and $so_{2n+1}$}, Math. Proc. Cambridge Philos. Soc. 126
(1999), 223-257].
\end{abstract}

\maketitle
\section{Introduction and main results}
The modular representations of reductive Lie algebras in prime characteristic
have been developed over the past decades with intimate connections to algebraic groups (cf. \cite{KW}, \cite{FP1, FP4}, \cite{Jan1, Jan2}, \cite{Pr}, \cite{LS}, \cite{YSL} etc.).

Let $k$ be an algebraically closed field of prime characteristic $p$
and $G$ be a connected, reductive algebraic group over $k$ with
$\ggg=\hbox{Lie}(G)$. Fix a maximal torus $T$ of $G$ and let $X(T)$
be the character group of $T$. Assume that the derived group
$G^{(1)}$ of $G$ is simply connected, $p$ is a good prime for the
root system of $\ggg$, and $\ggg$ has a non-degenerated
$G$-invariant bilinear form. Associated with any given linear form
$\chi\in\ggg^*$, the $\chi$-reduced enveloping algebra
$U_{\chi}(\ggg)$ is defined to be the quotient of the universal
enveloping algebra  $U(\ggg)$ by the ideal generated by all
$x^p-x^{[p]}-\chi(x)^p$ with $x\in \frak{g}$. Each isomorphism class
of irreducible representations of $\ggg$ corresponds to a unique
$p$-character $\chi$. Furthermore, a well-known result of
Kac-Weisfeiler shows that there is a Morita equivalence between
$U_{\chi}(\ggg)$-module category and $U_{\chi}(\frak l)$-module
category, where $\frak l$ is a certain reductive subalgebra of
$\ggg$ such that $\chi|_{[\frak l,\,\frak l]}$ is nilpotent (cf.
\cite{KW} and \cite{FP1}). This enables us to study
representations of $U_{\chi}(\ggg)$ just with nilpotent $\chi$.

We say a $p$-character $\chi$ has
standard Levi form if $\chi$ is nilpotent and if there exists a
subset $I$ of $\Pi$ such that $\chi (\ggg _{-\alpha})\neq 0$ for
$\alpha \in{I} $ and $\chi (\ggg _{-\alpha})= 0$ for $\alpha \in{R^+
\backslash I}$, where $\Pi$ is the set of all simple roots, $R^+$ is
the set of all positive roots and $|R^+|=N$ (cf. \cite[\S10]{Jan1}).
In this paper, we study representations of $\ggg$ with a $p$-character
$\chi$ of standard Levi form. Fix a triangular decomposition $\ggg=\frak n^-\oplus\frak
h\oplus\frak n^+$ and let $\frak b^+=\frak h\oplus\frak n^+$. For a subset $I$ in $\Pi$, set
$J=\Pi\setminus I$. Let $\ggg_J=\frak n_J^-\oplus\frak h\oplus\frak
n_J^+$ and $\frak p_J=\ggg_J\oplus \frak u_J^+$ be the Levi
subalgebra and the parabolic subalgebra of $\ggg$ corresponding to
$J$, respectively. Denote by $\widehat{L}_{\frak \ggg_J}(\lambda)$ the
irreducible $X(T)/\mathbb{Z}I$-graded $U_\chi(\ggg_J)$-module with
``highest" weight $\lambda$ (note that
$U_\chi(\ggg_J)=U_0(\ggg_J)$). Then $\widehat{L}_{\frak
\ggg_J}(\lambda)$ can be extended to a $U_\chi(\ppp_J)$-module with
trivial $\frak u_J^+$-action. The induced module
$U_\chi(\frak{g})\otimes_{U_\chi(\frak{p}_J)}\widehat{L}_{\frak
p_J}(\lambda)$ is called a parabolic baby Verma module, denoted by $\widehat{\cz}_P(\lambda)$.

In \cite[\S\,5.1]{FP4},  Friedlander and Parshall put forward the
following open question:
\begin{ques}\label{FP's Q} Can one give necessary and
sufficient conditions on an irreducible module for a parabolic
subalgebra $\ppp_J$ to remain irreducible upon induction to $\ggg$?
\end{ques}

When $\chi$ is regular nilpotent, Friedlander-Parshall
answered this question in \cite{FP1}. They showed that all such
inductions remain irreducible.  When $\ggg$ is of type $A_2$, and
$\chi(\neq0)$ is of standard Levi form, then each irreducible $\ggg$-module
with a $p$-character $\chi$ is a parabolic baby Verma module. Quite recently, the authors of the present paper  
obtained a necessary and sufficient condition for irreducibility of parabolic baby Verma modules of $\frak{sl}(4, k)$ in \cite{LSY}.

Let $C_0=\{\lambda\in X(T)_{\mathbb R}\mid
0\leq\langle\lambda+\rho,\alpha^{\vee}\rangle< p \mbox{ for all }\alpha\in
R^+ \}$ be the first dominant alcove of $X(T)_{\mathbb R}$. Let $X_1(T)=\{\lambda\in X(T)\mid
0\leq\langle\lambda+\rho,\alpha^{\vee}\rangle< p \mbox{ for all }\alpha\in
\Pi \}$ and $X'_1(T)\subset X_1(T)$ be a system of representatives
for $X(T)/pX(T)$. Each $\lambda\in X(T)$ has a unique
decomposition $\lambda=\lambda_0+p\lambda_1$ with $\lambda_0\in
X'_1(T)$ and $\lambda_1\in X(T)$ (cf. \cite[II \S\,9.14]{Jan3}). For
each $\lambda\in X(T)$, the map $\lambda\mapsto d\lambda$ induces a
bijection $X(T)/pX(T)\cong \Lambda=\{\mu\in\hhh^*\mid
\mu(h)^p-\mu(h^{[p]})=0,\,\forall\,h\in\hhh\}$ (cf. \cite[\S\,11.1]{Jan1}). So
we can regard $X'_1(T)$ as a system of representatives for
$\Lambda$. We call $\lambda\in X(T)$ $p$-regular if the
stabilizer of $\lambda$ in $W_p$ is trivial, where $W_p$ is the affine Weyl group of $\ggg$.

In the present paper, we give a sufficient condition on the irreducibility
of some parabolic baby Verma modules, which partially answers
Question \ref{FP's Q}.

\begin{thm}\label{main thm}
Let $\frak g$ be of type $A_n, B_n, C_n$ or $D_n$ satisfying the
hypotheses (H1)-(H3) in Section 2.1. Let $\lambda=\lambda_0+p\lambda_1\in X(T)$ such that $\lambda_0\in
C_0$ and $\lambda_1\in X(T)$. Let $\chi\in\ggg^*$ be of
standard Levi form and  $I=\{\alpha\in\Pi\mid
\chi(\ggg_{-\alpha})\neq 0\}$ in the Dynkin diagram of $\ggg$ is one of
the following forms
\begin{equation*}
\begin{cases}
\text{\rm(i)\,\,or\,\, (ii)}, & \text{\rm for}\,\, \ggg=A_n,\\
\text{\rm(iii)}, & \text{\rm for}\,\,\ggg=B_n,\\
\text{\rm(iv)}, & \text{\rm for}\,\, \ggg=C_n,\\
\text{\rm(v)\,\,or\,\, (vi)}, & \text{\rm for}\,\,\ggg=D_n,\\
\end{cases}
\end{equation*}
where
$${\rm (i)}\quad
\underbrace{\bullet-\cdot\cdot\cdot-\bullet-\bullet}\limits_{I}-\circ-\cdot\cdot\cdot-\circ,$$\vspace{1mm}

$${\rm(ii)}\quad \circ-\circ-\cdot\cdot\cdot-\underbrace{\bullet-\cdot\cdot\cdot-\bullet-\bullet}\limits_{I}.$$\vspace{1mm}

$${\rm (iii)}\quad\circ-...-\circ-\underbrace{\bullet-...-\bullet\Rightarrow\bullet}_{I}.$$\vspace{1mm}

$${\rm (iv)}\quad\underbrace{\bullet-\bullet-...-\bullet}_{I}-\circ-...-\circ\Leftarrow\circ.$$\vspace{1mm}

\[{\rm (v)}\quad\begin{matrix}\circ&\cdot\cdot\cdot&-\circ-&\bullet&-
&\cdot\cdot\cdot&-&\bullet&-&\bullet\\&&&&&&&|&&\\&&&&&&&\bullet&&\end{matrix}\]\vspace{1mm}

\[{\rm (vi)}\quad\begin{matrix}\bullet-&\cdot\cdot\cdot&-&\bullet&
-&\cdot\cdot\cdot&-&\bullet&-&\bullet\\&&&&&&&|&&\\&&&&&&&\circ&&\end{matrix}\]
Set $J=\Pi\setminus I$. Then the parabolic baby Verma module
$\widehat{\cz}_P(\lambda)$ is irreducible, provided that $\lambda$ is $p$-regular.
\end{thm}

As a consequence of Theorem \ref{main thm}, we have

\begin{cor}\label{cor1} Maintain the notations as in Theorem \ref{main thm}. The following statements hold.
\begin{itemize}
\item[(1)] Assume that $\frak g$ is  of type  $A_n$, and that $\chi\in\ggg^*$ is
sub-regular nilpotent and has standard Levi form. Then each
irreducible $\ggg$-module is a parabolic baby Verma module. 

\item[(2)] Assume that $\frak g$ is of type $B_n$, and that $\chi\in\ggg^*$ is
sub-regular nilpotent and has standard Levi form. Let $\{\widehat
L_\chi(\lambda_i)\mid 1\leq i\leq 2n\}$ be the set of isomorphism
classes of simple $\ggg$-modules in the block containing $\widehat
L_\chi(\lambda_1)$ described as in \cite[Proposition 3.13]{Jan4}.
Then $\widehat{\cz}_P(\lambda_i)$, $i\neq n,2n$, is irreducible with dimension
$r_{i}p^{N-1}$ for $1\leq i\leq n-1$ and $r_{2n-i}p^{N-1}$ for
$n+1\leq i\leq 2n-1 $.
\end{itemize}
\end{cor}

\begin{rem} Corollary \ref{cor1} coincides with the results by Jantzen in \cite[Theorem 2.6, Proposition 3.13]{Jan4}.
\end{rem}

\section{Preliminaries}
\subsection{Notations and assumptions}
Throughout this paper, we always assume that $k$ is an algebraically
closed field of prime characteristic $p$.  We use notations in
\cite{Jan1}.

 Let $G$ be a connected, reductive algebraic group over $k$ and $\ggg=\hbox{Lie}(G)$.
 Then $\ggg$ carries a natural restricted mapping $[p]$: $x\mapsto x^{[p]}$.
We assume that the following
 three hypotheses are  satisfied
(\cite[\S\,6.3]{Jan1}):
 \begin{enumerate}
  \item [(H1)] The derived group $\mathcal DG$ of $G$ is simply connected;
  \item [(H2)] The prime $p$ is good for $\ggg$;
  \item [(H3)] There exists a $G$-invariant non-degenerate bilinear form on
  $\ggg$.
\end{enumerate}

Let $T$ be a maximal torus of $G$ and $X(T)$ be the character group
of $T$. Denote respectively by $R^\pm$ the sets of all positive
roots and all negative roots. For each $\alpha \in R$, let
$\ggg_\alpha$ denote the root subspace of $\ggg$ corresponding to
$\alpha$ and $\frak n^+=\sum_{\alpha \in R^+}\ggg_\alpha, \frak
n^-=\sum_{\alpha \in R^-}\ggg_\alpha$. We have the triangular
decomposition: $\ggg=\frak n^+\oplus \frak h\oplus\frak n^-$ with
$\frak h$ being the Cartan subalgebra of $\ggg$ with rank $l$. Take
a Chevalley basis $\{x_{\alpha}, y_{\alpha}, h_i\mid \alpha\in R^+,
1\leq i\leq l\}$ of $\ggg$. Let $\frak b^+=\frak h\oplus \frak n^+$
be the Borel subalgebra of $\ggg$. For each $\alpha\in R$, let
$\alpha ^{\vee}$ denote the coroot of $\alpha$, $W$ the Weyl group
generated by all $s_\alpha$ with $ \alpha \in R $, and $W_p$ the
affine Weyl group generated by $s_{\alpha,r}\, (r\in \mathbb{Z})$,
where $s_{\alpha,r}$ is the affine reflection defined by
$s_{\alpha,r}(\mu)=\mu-(\langle\mu,\alpha^{\vee}\rangle-rp)\alpha$
for any $\mu\in\frak{h}^*$. Define the dot action of $w$ on
$\lambda$ by $w.\lambda=w(\lambda+\rho)-\rho$ for $w\in W$ and
$\lambda\in\hhh^*$, where $\rho$ is half the sum of all positive
roots.

\subsection{Baby Verma modules} Modulo Morita equivalence of representations,
we can assume that $\chi(\bbb^+)=0$ without loss of generality (cf. \cite{KW, FP1}). Set
$\Lambda:=\{\lambda \in \hhh^*\mid \lambda(h)^p=
\lambda(h^{[p]})\}$. Any simple $U_0(\hhh)$-module corresponds to a
unique $\lambda\in\Lambda$ (cf. \cite{Jan1}) and is one-dimensional,
denoted by $k_\lambda=kv_{\lambda}$, with $h\cdot
v_{\lambda}=\lambda(h)v_{\lambda}$ for any $h\in \hhh$. Since
$k_\lambda$ can be extended to a $U_0(\bbb^+)$-module with trivial
$\nnn^+$-action, we have an induced module
${Z}_{\chi}(\lambda)=U_{\chi}(\ggg)\otimes_{U_{0}(\frak b^+)}
k_{\lambda}$ which is called a baby Verma module. Each simple
$U_{\chi}(\ggg)$-module is the homomorphic image of some baby Verma
module ${Z}_{\chi}(\lambda),\lambda\in\Lambda$ (cf. \cite{Jan1} or
\cite{Jan2}).

\subsection{Standard Levi forms}\label{SLF}  We say a
$p$-character $\chi$ has standard Levi form if $\chi$ is nilpotent
and if there exists a subset $I$ of all simple roots such
that

\begin{equation}
 \chi (\ggg _{-\alpha})=\begin{cases}
      \neq 0, & \text{if $\alpha \in{I} $,}\\
      0,  & \text{if $\alpha \in{R^+ \backslash I}$}.
    \end{cases}
\end{equation}
As in \cite[\S\,10.4; \S\,10.5]{Jan1}, when $I$ is the full set of
all simple roots, we call $\chi$ a regular nilpotent element in
$\ggg^*$. When $I=\varnothing$, then $\chi=0$. We denote by $R_I$
the root system corresponding to the subset $I$, and
$W_I$ the Weyl group generated by all the $s_\alpha$ with $\alpha \in I$.

\subsection{Graded module category} Assume that $\chi$ is of standard Levi form with $I=\{\alpha\in\Pi\mid
\chi(\ggg_{-\alpha})\neq 0\}$. Following Jantzen \cite[\S\,11]{Jan1}, we can define a refined subcategory of the $U_\chi(\ggg)$-module category, which is the $X(T)/\mathbb ZI$-graded $U_\chi(\ggg)$-module
category, denoted by $\mathcal C$. For $\lambda\in
X(T)$, the graded baby Verma module $\widehat{{Z}}_{\chi}(\lambda)$ has a unique irreducible quotient, denoted by $\widehat{L}_\chi(\lambda)$. The latter is also a graded simple module.

\subsection{Parabolic baby Verma module}\label{S2.5}Assume that $\chi$ has standard Levi form
associated with a subset $I$ of the full set $\Pi$ of simple roots.
Set $J=\Pi\setminus I$. Let $\frak g_J=\frak h \oplus
\bigoplus_{\alpha\in R\cap\mathbb ZJ}\frak g_{\alpha}$, $\frak
u_J^+=\bigoplus_{\alpha>0,\alpha\notin \mathbb ZJ}\frak g_{\alpha}$, and
$\frak p_J=\frak g_J\oplus \frak u_J^+$. Then $U_\chi(\frak
p_J)=U_0(\frak p_J)$.  For $\mu\in X(T)$, let $\widehat L_{\frak
p_J}(\mu)$ be the graded irreducible $U_0(\frak p_J)$-module, which
is indeed a graded irreducible $U_0(\frak \ggg_J)$-module with
trivial $u_J^+$-action. The parabolic baby Verma module is
defined as the following induced module
$$\widehat \cz_P(\lambda):=U_{\chi}(\frak g)\otimes_{U_{0}(\frak
p_J)}\widehat L_{\frak p_J}(\lambda), \lambda\in X(T).$$ Then
$\widehat \cz_P(\lambda)$ is a quotient module of $\widehat
Z_\chi(\lambda)$. Let $\varphi$: $\widehat
Z_\chi(\lambda)\twoheadrightarrow \widehat \cz_P(\lambda)$ be the
canonical surjective morphism.

\subsection{}We apply the translation principle to $\widehat \cz_P(\lambda)$.

\begin{lem}\label{T1} Let $\lambda=(0,0,...,0)$ and $\mu\in C_0\cap X(T)$ be a regular weight.
Then we have $$T_\lambda^\mu\widehat \cz_P(\lambda)=\widehat
\cz_P(\mu)$$ where $T_{\lambda}^{\mu}$ is the so-called translation
functor defined in \cite[\S\,11]{Jan1}.
\end{lem}
\begin{proof}
By the definition of the translation factor, we have
$$T_\lambda^\mu\widehat \cz_P(\lambda)=\pr_\mu(L(\nu)\otimes \widehat
\cz_P(\lambda))$$ where $L(\nu)$ is the simple $G$-module with
highest weight $\nu$ in $W(\mu-\lambda)$ (cf.
\cite[\S\,11.20]{Jan1}).

Since $\widehat L_\chi(\lambda)$ is the head of $\widehat
\cz_P(\lambda)$ and $T_\lambda^\mu\widehat
L_\chi(\lambda)\cong\widehat L_\chi(\mu)$ (cf. \cite[Proposition
11.21]{Jan1}), we have $T_\lambda^\mu\widehat \cz_P(\lambda)\neq 0.$

Let $\Xi$ be the set of weights in $L(\nu)$. By \cite[Lemma 7.7,
Proposition 7.11]{Jan3}, there exists a unique weight $\xi\in \Xi$
with $\xi+\lambda\in W_p.\mu$ and $\xi+\lambda=\mu$. Note that
$\lambda$ is trivial, then $\xi=\mu$. Hence, $\mu$ is the unique
weight of $L(\nu)\otimes k_\lambda$ lying in $W_p.\mu$, where
$k_\lambda$ is the one-dimensional trivial $U_0(\ggg_J)$-module.

The generalized tensor identity (cf. \cite[\S1.12]{Jan2}) yields the
following isomorphism
\begin{equation}\label{F1}
T_\lambda^\mu\widehat \cz_P(\lambda):=\pr_\mu(L(\nu)\otimes \widehat
\cz_P(\lambda))\cong \pr_\mu\big(U_{\chi}(\frak
g)\otimes_{U_{0}(\frak p_J)}(L(\nu)\otimes \widehat L_{\frak
p_J}(\lambda))\big).
\end{equation}

Since $\lambda=(0,0,...,0)$, the simple module $\widehat L_{\frak p_J}(\lambda)$ is
trivial, i.e., $\widehat L_{\frak p_J}(\lambda)=k_\lambda$. Then
$L(\nu)\otimes \widehat L_{\frak p_J}(\lambda)=L(\nu)\otimes
k_\lambda.$ The set of weights in $L(\nu)\otimes k_{\lambda}$ is just
$\Xi$.

We have the following composition series of $L(\nu)\otimes
k_\lambda$
$$0=M_0\subset M_1\subset M_2\subset \cdot\cdot\cdot\subset
M_r=L(\nu)\otimes k_\lambda$$ where the factors $M_j/M_{j-1}\cong
\widehat L_{\frak p_J}(\lambda_j)$ with $\lambda_j\in \Xi$. Since
$T_\lambda^\mu\widehat \cz_P(\lambda)\neq 0$ and $\mu$ is the unique
weight of $L(\nu)\otimes \lambda$ which lies in $W_p.\mu$, there
exists a unique $l\leq r$ with $\lambda_l=\mu$. Then $\widehat
L_{\frak p_J}(\mu)\cong M_l/M_{l-1}$ is the unique composition
factor of $L(\nu)\otimes \lambda$ whose highest weight lies in
$W_p.\mu$.

Since the short sequence $0\rightarrow M_l\rightarrow
M_r\rightarrow M_r/M_l\rightarrow0$ is exact and the functor
$\pr_\mu\big( U_{\chi}(\frak g)\otimes_{U_{0}(\frak p_J)}(-)\big)$ is
exact, we have the following isomorphisms
\begin{align}\label{F2}
&\pr_\mu\big(U_{\chi}(\frak g)\otimes_{U_{0}(\frak p_J)}(L(\nu)\otimes
\widehat L_{\frak p_J}(\lambda))\big)\\
=&\pr_\mu\big(U_{\chi}(\frak
g)\otimes_{U_{0}(\frak p_J)}M_r\big)\cr \cong&
\pr_\mu\big(U_{\chi}(\frak g)\otimes_{U_{0}(\frak p_J)}M_l\big)\cr
\cong & \pr_\mu\big(U_{\chi}(\frak g)\otimes_{U_{0}(\frak
p_J)}\widehat L_{\frak p_J}(\mu)\big)\cr\cong&\pr_\mu\big(\widehat \cz_P(\mu)\big).
\end{align}

As $\widehat \cz_P(\mu)$ is the quotient of the baby Verma module
$\widehat Z_\chi(\mu)$, then $\widehat \cz_P(\mu)$ has a simple head
and is indecomposable. Furthermore, we have
\begin{equation}\label{F3}
\pr_\mu\big(\widehat \cz_P(\mu)\big)\cong
\widehat \cz_P(\mu).
\end{equation}
It follows from (\ref{F1}), (\ref{F2}) and (\ref{F3}) that $T_\lambda^\mu\widehat \cz_P(\lambda)=\widehat \cz_P(\mu).$
The proof is completed.
\end{proof}

Let $\widehat M\in \mathcal C$. A weight vector $m\in \widehat M$ is
called a {\sl maximal weight vector} if $x_{\alpha}.m=0$ for any
$\alpha\in R^+$.

\subsection{}We need the following lemma for later use.
\begin{lem}\label{lem1}Assume that $\chi(\frak b^+)=0$. Then every submodule of
$\widehat \cz_P(\lambda)$ contains a {\sl maximal weight vector}.
\end{lem}
\begin{proof}Let $\widehat M$ be a submodule of $\widehat \cz_P(\lambda)$.
Since $U_0(\frak b^+)$ is a subalgebra of $U_\chi(\frak g)$, $\widehat
M$ is also a $U_0(\frak b^+)$-module. Then there exists an irreducible
$U_0(\frak b^+)$-submodule in $\widehat M$ which is one-dimensional
annihilated by $\nnn^+$. This completes the proof.
\end{proof}

\section{Proof of Theorem \ref{main thm} for type $A_n$}

\subsection{}In this section, we always assume that $\frak g$ is a simple Lie
algebra of type  $A_n$ and $\chi\in\ggg^*$ has standard Levi form, and $I=\{\alpha\in\Pi\mid
\chi(\ggg_{-\alpha})\neq 0\}$ in the Dynkin diagram is described as case (i) in Theorem \ref{main thm}.

We assume that $I=\{\alpha_1,\alpha_2,...,\alpha_s\}$ with $s\leq n$.
Then $J=\{\alpha_{s+1},...,\alpha_n\}$. If $s= n$, i.e., $\chi$ is
regular nilpotent, it's well known that $\widehat \cz_P(\lambda)= \widehat
Z_\chi(\lambda)$ is irreducible (cf. \cite[\S\,10]{Jan1}). When
$s<n$, we have
\begin{claim}\label{a claim} The parabolic baby Verma module $\widehat \cz_P(\lambda)$ has only one
maximal weight vector (up to scalars) that generate $\widehat \cz_P(\lambda)$.
\end{claim}

\begin{rem}\label{remark for type A}
It follows from Claim \ref{a claim} and Lemma \ref{lem1} that
$\widehat \cz_P(\lambda)$ is irreducible.
\end{rem}

In the following subsections, we first prove Claim \ref{a claim}.

\subsection{} \label{TrivalLJ}
Thanks to Lemma \ref{T1}, it suffices to prove Claim \ref{a claim} for the special case that $\lambda=(0,0,...,0)$. In this case, the irreducible $\widehat {L}_{\ppp_J}(\lambda)$-module is one-dimensional. This means that this module is $\ppp_J$-trivial. Then $\widehat \cz_P(\lambda)=U_\chi(\uuu_J^-)$ as vector spaces, where $\uuu_J^-$ is the negative counterpart of $\uuu_J^+$ such that $\ggg=\uuu_J^-\oplus \ppp_J$.

\subsection{} In the remainder of  this section, we always take $\lambda=(0,0,\cdots,0)$.
Assume that $w=u\otimes v_\lambda$ is a maximal weight vector of $\widehat \cz_P(\lambda)$ of weight $\mu$. We aim at proving that $w$ generates the whole $\widehat \cz_P(\lambda)$.

Suppose the submodule generated by $w$ is proper, then $\mu\in W_p.\lambda$, and $\mu=\lambda-\sum\limits_{i=1}^{n} k_i\alpha_i,~k_i\in\mathbb{Z}_+$
(cf. \cite[Proposition 4.5]{Jan2}). Fix an order of the Chevalley basis
in $\uuu_J^-$ as follows
\begin{align*}
&y_{\alpha_1}; \\
&y_{\alpha_1+\alpha_2}, y_{\alpha_2};\\
&y_{\alpha_1+\alpha_2+\alpha_3},y_{\alpha_2+\alpha_3},y_{\alpha_3};\\
&\cdots\\
&y_{\alpha_1+\alpha_2+\cdots+\alpha_s},\cdots,  y_{\alpha_{s-1}+\alpha_s},y_{\alpha_s};\\
&y_{\alpha_1+\alpha_2+\cdots+\alpha_{s+1}},\cdots,y_{\alpha_s+\alpha_{s+1}};\\
&\cdots\\
&y_{\alpha_1+\alpha_{2}+\cdots+\alpha_n},\cdots, y_{\alpha_s+\cdots+\alpha_{n}}.
\end{align*}

Then $U_\chi(\uuu_J^-)$ has the following basis
\begin{align}\label{Uminus}
{\underline y}_1^{\bf a_1}\underline
y_2^{\bf a_2} \cdots \underline y_s^{\bf a_s} \underline y_{s+1}^{\bf a_{s+1}}\cdots \underline y_n^{\bf a_n},
\end{align}
where $\bf a=(a_1,a_2,\cdots,a_n)$,
$$\underline y_j^{\bf
a_j}=y_{\alpha_{1}+\alpha_2+...+\alpha_j}^{a_j(1)} y_{\alpha_2+...+\alpha_j}^{a_j(2)}\cdots y_{\alpha_{j-1}+\alpha_j}^{a_j(j-1)}y_{\alpha_j}^{a_j(j)}\cdot\cdot\cdot
\mbox{ for }
j=1,2,\cdots,s$$
and
$$\underline y_j^{\bf
a_j}=y_{\alpha_{1}+\cdots+\alpha_j}^{a_j(1)}y_{\alpha_{2}+\cdots+\alpha_j}^{a_j(2)}\cdot\cdot\cdot
y_{\alpha_{s}+\alpha_{s+1}+\cdots+\alpha_j}^{a_j(s)} \mbox{ for }
j=s+1,\cdots,n$$
with $0\leq a_j(k)\leq p-1,\,\forall\,j, k$. Hence $u\otimes v_\lambda$ can be uniquely written as follows
\begin{align}\label{eq3.1}
u\otimes v_\lambda=\sum\limits_{\bf a}l_{\bf a}{\underline y}_1^{\bf a_1}\underline
y_2^{\bf a_2} \cdots \underline y_s^{\bf a_s} \underline y_{s+1}^{\bf a_{s+1}}\cdots \underline y_n^{\bf a_n}\otimes v_\lambda.
\end{align}

For $\widehat M\in\mathcal C$, we have a decomposition $\widehat
M=\bigoplus\limits_{\nu\in X(T)/\mathbb ZI}\widehat M^{\nu}$ with
$x_\alpha.\widehat M^{\nu}\subset \widehat M^{\nu+\alpha}$ for all
$\alpha\in R$.  By \cite[\S\,11.7]{Jan1}, we have $\widehat
Z_\chi(\lambda)^{\lambda+\mathbb
ZI}\cong_{\ggg_I} \widehat
L_\chi(\lambda)^{\lambda+\mathbb ZI}$. Hence,
$\widehat \cz_P(\lambda)^{\lambda+\mathbb ZI}\cong_{\ggg_I}\widehat
L_\chi(\lambda)^{\lambda+\mathbb ZI}.$
Since $u\otimes v_\lambda$
generates a proper submodule of $\widehat \cz_P(\lambda)$, it follows that
 \begin{align}\label{Aassumption}
u\otimes v_\lambda\notin\widehat
\cz_P(\lambda)^{\lambda+\mathbb ZI} \mbox{ and }
\mu\in
W_p.\lambda\setminus W_{I,p}.\lambda.
\end{align}


\subsection{} It follows from (\ref{Uminus}) and (\ref{eq3.1}) that $k_s\geq k_{s+1}$ for a weight vector $u\otimes v_\lambda\in \widehat \cz_P(\lambda)$ of weight $\mu=\lambda-\sum\limits_{i=1}^{n} k_i\alpha_i,~k_i\in\mathbb{Z}_+$. Furthermore, we have
\begin{lem}\label{lem2} Assume that $u\otimes v_\lambda$ is a {\sl maximal weight vector} of $\widehat \cz_P(\lambda)$ with weight
$\mu=\lambda-\sum\limits_{i=1}^{n} k_i\alpha_i,~k_i\in\mathbb{Z}_+$,
and that $u\otimes v_\lambda$ generates a proper submodule of $\widehat \cz_P(\lambda)$. Then $k_s>k_{s+1}$.
\end{lem}
\begin{proof}
Suppose $k_{s}=k_{s+1}\neq0$. Then
the factor $\underline y_{s}^{\bf a_{s}}$ does not appear in any
monomial summand of (\ref{eq3.1}). The expression (\ref{eq3.1}) can be written as
\begin{align}\label{eq3.2}
u\otimes v_\lambda=&\sum\limits_{\bf a}l_{\bf a}{\underline
y}_1^{a_1(1)}\underline y_2^{(a_2(1), a_2(2))}
\cdot\cdot\cdot\underline
y_{s-1}^{(a_{s-1}(1),a_{s-1}(2),\cdots,a_{s-1}(s-1))}\\
&\cdot\underline
y_{s+1}^{(a_{s+1}(1),a_{s+1}(2),\cdots,a_{s+1}(s))}
\cdot\cdot\cdot\underline
y_n^{(a_n(1),a_n(2),\cdots,a_n(s))}\otimes
v_{\lambda}.\nonumber
\end{align}

Since $k_{s+1}\neq0$, there exists some factor $\underline y_{i}^{\bf
a_{i}}\neq 0$ with $s+1\leq i\leq n$.  Without loss of generality,
we may assume that $\underline y_{s+1}^{\bf a_{s+1}}\neq 0$, then
\begin{align*}
& x_{\alpha_{s+1}}\cdot\underline
y_{s+1}^{(a_{s+1}(1),a_{s+1}(2),\cdots,a_{s+1}(s))}\cr =&
\sum\limits_{t=1}^{s} a_{s+1}(t)
N_{\alpha_{s+1},-(\alpha_t+\cdots+\alpha_{s+1})}y_{\alpha_t+\cdots+\alpha_{s}}
\underline
y_{s+1}^{(a_{s+1}(1),a_{s+1}(2),\cdots,a_{s+1}(t)-1,\cdots,a_{s+1}(s))}
\\
&+\underline y_{s+1}^{(a_{s+1}(1),a_{s+1}(2),\cdots,a_{s+1}(s))} x_{\alpha_{s+1}}
\end{align*}
where $N_{\alpha_{s+1},-(\alpha_t+\cdots+\alpha_{s+1})}$ is a
structure constant of $\ggg$ relative to the Chevalley basis. Since
$x_{\alpha_{s+1}}$ commutes with $\underline y_t$ for any $t$ with $t\neq s$, and annihilates $v_\lambda$,
it follows that
\begin{align}\label{eq3.3} & x_{\alpha_{s+1}}\cdot
u\otimes v_{\lambda}\\ =&(\sum\limits_{\bf a}
\sum\limits_{t=1}^{s}l_{\bf
a}a_{s+1}(t)N_{\alpha_{s+1},-(\alpha_t+...+\alpha_{s+1})}{\underline
y}_1^{a_1(1)}\cdot\cdot\cdot\underline
y_{s-1}^{(a_{s-1}(1),a_{s-1}(2),...,a_{s-1}(s))}
y_{\alpha_t+...+\alpha_{s}}\nonumber\\ &\cdot\underline
y_{s+1}^{(a_{s+1}(1),a_{s+1}(2),...,a_{s+1}(t)-1,...,a_{s+1}(s))}
\cdot\cdot\cdot \underline y_n^{(a_n(1),a_n(2),...,a_n(t),\cdots,a_n(s))})\otimes v_\lambda
.\nonumber
\end{align}
Note that $\cz_P(\lambda)$ is free over $U_\chi(\uuu_J^-)$. Since $u\otimes v_\lambda$ is a maximal weight vector, it follows from (\ref{eq3.3}) and (\ref{Uminus}) that $u\otimes v_\lambda=0$, a contradiction. Hence, $k_s>k_{s+1}$.
\end{proof}

\begin{lem}\label{lem3} Maintain the notations as in Lemma \ref{lem2}. Assume that
$$\underline y_s^{\bf
m_s}=y_{\alpha_1+\alpha_2+\cdots+\alpha_s}^{m_s(1)}
y_{\alpha_2+\cdots+\alpha_s}^{m_s(2)}\cdot\cdot\cdot
y_{\alpha_{s-1}+\alpha_s}^{m_s(s-1)}y_{\alpha_{s}}^{m_s(s)}$$ is a
factor of one monomial summand of $u$ such that $m_s(s)$ is maximal
among all the powers of $y_{\alpha_{s}}$ in monomial summands of
$u$. Then neither $y_{\alpha+\alpha_s}$ nor $y_{\alpha_s+\beta}$ with
$\alpha\in R_I^+$, $\beta\in R_J^+$, appear in a
monomial summand of $u$ containing the factor $y_{\alpha_{s}}^{m_{s}(s)}$.
\end{lem}

\begin{proof}
For the weight $\mu=\lambda-\sum\limits_{i=1}^{n}k_i\alpha_i$ of
$u\otimes v_\lambda$, we have $k_s>k_{s+1}$ by Lemma \ref{lem2}.
Hence, there exists some nontrivial factor $\underline y_{s}^{\bf a_s}$ in each monomial summand of $u$.

(1) We first claim that there
exists at least one monomial summand of $u$ which contains a factor
$y_{\alpha_s}^{m}$ with $m>0$.

Otherwise, the factor $~\underline
y_{s}^{\bf a_s}$ of each monomial summand of $u$ can be written as
$\underline y_{s}^{\bf
a_s}=y_{\alpha_1+\cdots+\alpha_s}^{a_s(1)}\cdot\cdot\cdot
y_{\alpha_{s-1}+\alpha_s}^{a_s(s-1)}$. Since $\bf
a_s\neq0$, there exist at least one $a_s(t)\neq0 $ for some $t\leq
s-1$. Consider the action of
$x_{\alpha_{t}+\cdots+\alpha_{s-1}}$ on $u\otimes v_\lambda$. By assumption, there does not
exist monomial summand of $u$ which contains a factor
$y_{\alpha_s}^{m}$ with $m>0$. Since $x_{\alpha_{t}+\cdots+\alpha_{s-1}}y_{\alpha_{t}+...+\alpha_s}=y_{\alpha_{t}+...+\alpha_s}x_{\alpha_{t}+\cdots+ \alpha_{s-1}}+N_{\alpha_{t}+\cdots+\alpha_{s-1},-(\alpha_{t}+...+\alpha_s)}y_{\alpha_s}$,
it follows that $x_{{\alpha_{t}+\cdots+\alpha_{s-1}}}\cdot u\otimes v_\lambda$ contains the following monomial summand
\begin{equation}\label{eq3.4}{\underline y}_1^{a_1(1)}\underline y_2^{(a_2(1),
a_2(2))} \cdot\cdot\cdot\underline
y_s^{(a_s(1),\cdots,a_s(t)-1,\cdots,a_s(s))}\cdots\underline
y_n^{(a_n(1),\cdots,a_n(s))}
\end{equation}
with $a_s(s)=1$. Note that there does not exist a similar item as (\ref{eq3.4}) among all the monomial summands of $x_{{\alpha_{t}+\cdots+\alpha_{s-1}}}\cdot u\otimes v_\lambda$, then $x_{\alpha_t+\cdots+\alpha_{s-1}}\cdot
u\otimes v_\lambda\neq 0$. This is a contradiction. So, there exists
at least one monomial summand of $u\otimes v_\lambda$ containing $y_{\alpha_s}^{m}$ with $m>0$.

(2) Suppose that
$y_{\alpha_{t}+...+\alpha_s}^{m_{s}(t)}\cdot\cdot\cdot y_{\alpha_{s-1}+\alpha_s}^{m_{s}(s-1)}y_{\alpha_{s}}^{m_{s}(s)}$
with $m_{s}(t)\geq 1$ is a factor of a monomial summand of $u\otimes v_\lambda$. Consider the action of $x_{\alpha_{t}+...+\alpha_{s-1}}$ on $u\otimes v_\lambda$. A
direct computation implies that
$y_{\alpha_{t}+...+\alpha_s}^{m_{s}(t)-1}\cdot\cdot\cdot y_{\alpha_{s-1}+\alpha_s}^{m_{s}(s-1)}y_{\alpha_{s}}^{m_{s}(s)+1}$
is a factor of a monomial summand of $x_{\alpha_{t}+...+\alpha_{s-1}}\cdot u\otimes
v_\lambda$. Since $m_s(s)$ is maximal
among all the powers of $y_{\alpha_{s}}$ in monomial summands of
$u$, there does not exist a similar item as
$y_{\alpha_{t}+...+\alpha_s}^{m_{s}(t)-1}\cdot\cdot\cdot y_{\alpha_{s-1}+\alpha_s}^{m_{s}(s-1)-1}y_{\alpha_{s}}^{m_{s}(s)+1}$
among all the monomial summands of $x_{{\alpha_{t}+\cdots+\alpha_{s-1}}}\cdot u\otimes v_\lambda$. Then $x_{\alpha_{t}+...+\alpha_{s-1}}\cdot u\otimes v_\lambda$ is not
zero and this is contrary to the fact that $u\otimes v_\lambda$ is
maximal. So $y_{\alpha_{t}+...+\alpha_s}^{m_s({t})}$ with
$m_{s}(t)\geq 1$ does not appear in the same monomial summand of
$u$ containing $y_{\alpha_{s}}^{m_{s}(s)}$, i.e., $y_{\alpha+\alpha_s}$ with $\alpha\in R_I^+$ do not appear in the
monomial summand of $u$ containing $y_{\alpha_{s}}^{m_{s}(s)}$. Similarly,
$y_{\alpha_s+\beta}$ with $\beta\in R_J^+$ do not appear in the
monomial summand of $u$ containing $y_{\alpha_{s}}^{m_{s}(s)}$.
\end{proof}

\subsection{Proof of Claim \ref{a claim}}
By Lemma \ref{lem3}, neither $y_{\alpha+\alpha_s}, \alpha\in R_I^+$ nor $y_{\alpha_s+\beta}, \beta\in R_J^+$ appear in the
monomial summand of $u$ provided it contains the factor $y_{\alpha_{s}}^{m_{s}(s)}$,
where $m_s(s)$ is maximal.

Denote by $u_s$ the sum of all those monomial summands of $u$
containing $y_{\alpha_{s}}^{m_{s}(s)}$. By Lemma \ref{lem3}, $u_s$
can be written as follows
\begin{align}\label{Exp}
u_s=&\sum\limits_{\bf a}l_{\bf
a}{\underline y}_1^{a_1(1)}\underline y_2^{(a_2(1), a_2(2))}
\cdot\cdot\cdot\underline
y_{s-1}^{(a_{s-1}(1),a_{s-1}(2),\cdots,a_{s-1}({s-1}))}
y_{\alpha_{s}}^{m_{s}(s)}\cr &\cdot
y_{s+1}^{(a_{s+1}(1),a_{s+1}(2),\cdots,a_{s+1}({s-1}))}
\cdot\cdot\cdot\underline
y_{n}^{(a_{n}(1),a_{n}(2),\cdots,a_{n}({s-1}))}.
\end{align}

If $k_{s-1}\neq 0$, by a similar argument as in the proof of Lemma
\ref{lem3},  there exists at least one
 monomial summand of $u$ which contains
$y_{\alpha_{s-1}}^{m}$ with $m>0$.  Let $m_{s-1}(s-1)$ be maximal
among all those $a_{s-1}(s-1)$ appearing in (\ref{Exp}).  Then there
exists a monomial summand of $u_s$ with $a_{s-1}(s-1)=m_{s-1}(s-1)$
which does not contain $y_{\alpha_{s}+\beta}$ and
$y_{\alpha+\alpha_{s}}$. By the same method as in the proof of Lemma
\ref{lem3}, those $y_{\alpha+\alpha_{s-1}},~\alpha\in R^+$ and
$y_{\alpha_{s-1}+\beta},~\beta\in R^+$ do not appear in the monomial
summand of $u_s$ in which $a_{s-1}(s-1)=m_{s-1}(s-1)$.

Denote by
\begin{align*}
u_{s-1}=& \sum\limits_{\bf a}l_{\bf a}{\underline
y}_1^{a_1(1)}\underline y_2^{(a_2(1), a_2(2))}
\cdot\cdot\cdot\underline
y_{s-2}^{(a_{s-2}(1),a_{s-2}(2),\cdots,a_{s-2}({s-2}))}
y_{\alpha_{s-1}}^{m_{s-1}(s-1)}y_{\alpha_{s}}^{m_{s}(s)}\cr
&\cdot\underline
y_{s+1}^{(a_{s+1}(1),a_{s+1}(2),\cdots,a_{s+1}({s-2}),0,0)}
\cdot\cdot\cdot\underline
y_{n}^{(a_{n}(1),a_{n}(2),\cdots,a_{n}({s-2}),0,\cdots,0)}.
\end{align*}
Similar to the discussion above, if $k_{s-2}\neq 0$,  then
$y_{\alpha+\alpha_{s-2}},~\alpha\in R^+$ and
$~y_{\alpha_{s-2}+\beta},~\beta\in R^+$, do not appear in the
monomial summand of  $u_{s-1}$  which contains
$y_{\alpha_{s-2}}^{m_{s-2}(s-2)}$, where $m_{s-2}(s-2)$ is maximal
among all those $a_{s-2}(s-2)$ appearing in (\ref{Exp}).

Since $I$ is connected, we can repeat the process above. Finally, we
obtain that $u_1= y_{\alpha_1}^{m_1({1})}y_{\alpha_{2}}^{m_2({2})}
\cdot\cdot\cdot y_{\alpha_{s-1}}^{m_{s-1}({s-1})}
y_{\alpha_{s}}^{m_s({s})}$ with $m_{s}({s})>0$ and
$m_{s-i}({s-i})\geq0,1\leq i\leq s-1$. As $u_1$ is a summand of
$u\otimes v_\lambda$, then $u\otimes v_\lambda\in\widehat
\cz_P(\lambda)^{\lambda+\mathbb ZI}\cong_{\ggg_I}\widehat
L_\chi(\lambda)^{\lambda+\mathbb ZI}$.  This is contrary to the fact
that $u\otimes v_\lambda$ generates the proper submodule of
$\widehat \cz_P(\lambda)$. Therefore the Claim \ref{a claim} holds
in the case that $I=\{\alpha_1,\cdots,\alpha_s\}$.

When $I=\{\alpha_t,\alpha_{t+1},\cdots,\alpha_n\},~1< t\leq n$, the
proof is similar. We complete the proof of Claim \ref{a claim}.

\subsection{}
For the decomposition $\lambda=\lambda_0+p\lambda_1$ with
$\lambda_0\in X'_1(T)$ and $\lambda_1\in X(T)$, we have $\mathcal
F(\widehat \cz_P(\lambda))\cong \cz_P(\lambda_0)$ where $\mathcal F$ is
the forgetful functor (cf. \cite[\S11]{Jan1}). So the non-graded
parabolic baby Verma module $\cz_P(\lambda_0)$ is also irreducible.

\subsection{}
Assume that $I=\{\alpha_1,\alpha_2,\cdots,\alpha_s\}$ with $s\leq n$.
From the discussions above, if one replaces the condition
$\lambda=(0,0,\cdots,0)$ by $(m_{1}-1,\cdots,m_{s}-1,0,\cdots,0)$ with
$0<m_i<p, \forall~ i$, Lemma \ref{T1} and Claim \ref{a claim} still
hold. We have the following consequence.
\begin{cor} Keep the same notations as above. Assume that
$I$ is connected in the Dykin diagram with $\alpha_1\in I$ (resp.
$\alpha_n\in I$) and $\lambda_0+\rho=(m_{1},\cdots,m_{s},1,\cdots,1)$
(resp. $\lambda_0+\rho=(1,\cdots,1,m_{t},\cdots,m_{n})$), $0<m_i<p,
\forall~ i$. Then   $\widehat \cz_P(\lambda)$ is irreducible.
\end{cor}

\begin{rem}\label{rem3.6}For $\lambda=\lambda_0+p\lambda_1$ with
$\lambda_0\in X'_1(T)$ and $\lambda_1\in X(T)$, by Lemma \ref{T1}, when $\lambda_0$ lies in the
alcoves which contain the weight $(m_{1},\cdots,m_{s},1,\cdots,1)-\rho$
(resp. $(1,\cdots,1,m_{t},\cdots,m_{n})-\rho$), $0<m_i<p, \forall~ i$,
the parabolic baby Verma module $\widehat \cz_P(\lambda)$ is
irreducible.
\end{rem}

\section{Proof of Theorem \ref{main thm} for types $B_n$, $C_n$ and $D_n$}
In this section, we give the proof of Theorem \ref{main thm} for types
$B_n$, $C_n$ and $D_n$ cases by cases.
\subsection{Proof of Theorem \ref{main thm} for type $B_n$} Let $\frak
g$ be of type $B_n$ and $\chi\in\ggg^*$ be of standard Levi form
with $I=\{\alpha\in\Pi\mid
\chi(\ggg_{-\alpha})\neq 0\}$. Assume that $I=\{\alpha_s,\cdots,\alpha_n\}$ for
some $1\leq s\leq n$ (note that $\alpha_n$ is a short root), i.e.,
we have the following Dynkin diagram
$$\circ-\cdots-\circ-\underbrace{\bullet-\cdots-\bullet\Rightarrow\bullet}_{I}$$

Fix an order of the Chevalley basis
in $\uuu_J^-$ as follows
\begin{align*}
&y_{\alpha_n}; \\
&y_{\alpha_{n-1}+\alpha_{n}},y_{\alpha_{n-1}+2\alpha_{n}},y_{\alpha_{n-1}};\\
&\cdots\\
&y_{\alpha_{s}+\alpha_{s+1}},y_{\alpha_{s}+\alpha_{s+1}+\alpha_{s+2}},\cdots,y_{\alpha_{s}+2\alpha_{s+1}+\cdots+2\alpha_{n-1}+2\alpha_{n}}, y_{\alpha_{s}} ;\\
&y_{\alpha_{s-1}+\alpha_{s}},y_{\alpha_{s-1}+\alpha_{s}+\alpha_{s+1}},\cdots,y_{\alpha_{s-1}+2\alpha_{s}+\cdots+2\alpha_{n-1}+2\alpha_{n}};\\
&\cdots\\
&y_{\alpha_{1}+\alpha_{2}+\cdots+\alpha_{s}},\cdots, y_{\alpha_{1}+\alpha_{2}+\alpha_{3}+\cdots+\alpha_{n}} \cdots,y_{\alpha_{1}+\alpha_{2}+2\alpha_{3}+\cdots+2\alpha_{n}},y_{\alpha_{1}+2\alpha_{2}+2\alpha_{3}+\cdots+2\alpha_{n}}.
\end{align*}

Suppose that $u\otimes v_\lambda$ is a {\sl maximal weight vector}
of $\widehat \cz_P(\lambda)$  such that the submodule of $\widehat
\cz_{P}(\lambda)$ generated by $u\otimes v_\lambda$ is proper. Then
the weight $\mu$ of  $u\otimes v_\lambda$ belongs to $W_p.\lambda$
and can be written as $\mu=\lambda-\sum\limits_{i=1}^{n}
k_i\alpha_i,~k_i\in\mathbb{Z}_+$ (cf. \cite[Proposition 4.5]{Jan2}).
We can use similar discussion as the case of type $A_n$.
We enumerate the strategy as follows with the details omitted.

(i) Assume that $\lambda=(0,0,\cdots,0)$. Then $k_s>0$.

(ii) There exists $y_{\alpha_{s}}^{m}$ with $m>0$ as a factor of
some monomial summand of $u$. Assume that the power $m_{s}(s)$  of $y_{\alpha_{s}}$ is  maximal among all the powers of $y_{\alpha_{s}}$ in monomial summands of
$u$. First, we prove that
$y_{\alpha_{s+1}+\alpha_s}$ do not appear in the
same monomial summand of $u$ which contains $y_{\alpha_{s}}^{m_{s}(s)}$. Next,
we can prove that $y_{\gamma}$ does not appear in the same monomial
summand of $u$ which contains $y_{\alpha_{s}}^{m_{s}(s)}$ as a factor,
where $\gamma\in R^+$ and $\alpha_{s}$ is the first or
last summand of $\gamma$.

(iii) $u$ has a monomial summand
$y_{\alpha_n}^{m_n({n})}y_{\alpha_{n-1}}^{m_{n-1}({n-1})}
\cdot\cdot\cdot y_{\alpha_{s+1}}^{m_{s+1}({s+1})}
y_{\alpha_{s}}^{m_s({s})}.$

Furthermore, we have
\begin{rem}\label{reamrk B}Let $\lambda=\lambda_0+p\lambda_1$ with
$\lambda_0\in X'_1(T)$ and $\lambda_1\in X(T)$. Assume that $\lambda_0$ lies in the
alcoves which contain the weight
$(1,\cdots,1,m_{s},\cdots,m_{n})-\rho$, $0<m_i<p, \forall~ i$,
then the parabolic baby Verma module $\widehat \cz_{P}(\lambda)$ is
irreducible by Lemma \ref{T1}.
\end{rem}

\subsection{Proof of Theorem \ref{main thm} for type $C_n$}Let $\frak
g$ be of type $C_n$ and $\chi\in\ggg^*$ be of standard Levi form
with $I=\{\alpha\in\Pi\mid
\chi(\ggg_{-\alpha})\neq 0\}$. Assume that $I=\{\alpha_1,\cdots,\alpha_s\}, 1\leq
s\leq n$ (note that $\alpha_n$ is a long root), i.e., we have the
following Dynkin diagram

$$\underbrace{\bullet-\bullet-\cdots-\bullet}_{I}-\circ-\cdots-\circ\Leftarrow\circ$$

Fix an order of the Chevalley basis
in $\uuu_J^-$ as follows
\begin{align*}
&y_{\alpha_1}; \\
&y_{\alpha_{1}+\alpha_{2}},y_{\alpha_{2}};\\
&\cdots\\
&y_{\alpha_{s-1}+\alpha_{s}},y_{\alpha_{s-2}+\alpha_{s-1}+\alpha_{s}},\cdots,y_{\alpha_{1}+\alpha_{2}+\cdots+\alpha_{s}}, y_{\alpha_{s}} ;\\
&\cdots\\
&y_{\alpha_{s}+\alpha_{s+1}+\cdots+\alpha_{n-1}}, y_{\alpha_{s-1}+\alpha_{s}+\cdots+\alpha_{n-1}},\cdots,y_{\alpha_{1}+\alpha_{2}+\cdots+\alpha_{n-1}};\\
&y_{\alpha_{s}+\alpha_{s+1}+\cdots+\alpha_{n}},\cdots, y_{\alpha_{1}+\alpha_{2}+\cdots+\alpha_{n}} \cdots,y_{\alpha_{1}+2\alpha_{2}+\cdots+2\alpha_{n-1}+\alpha_{n}},y_{2\alpha_{1}+2\alpha_{2}+\cdots+2\alpha_{n-1}+\alpha_{n}}.
\end{align*}
By a similar argument as the case of type $B_n$, we
can prove Theorem \ref{main thm} in the case of type $C_n$.

Furthermore, let $\lambda=\lambda_0+p\lambda_1,\lambda_0\in X'_1(T),\lambda_1\in X(T)$. Assume that $\lambda_0$ lies in the
alcoves which contain the weight
$(m_{1},\cdots,m_{s},1,\cdots,1)-\rho$, $0<m_i<p, \forall~ i$. Then the parabolic baby Verma module $\widehat \cz_{P}(\lambda)$ is
irreducible by Lemma \ref{T1}.

\subsection{Proof of Theorem \ref{main thm} for type $D_n$} Let $\frak g$ be of type $D_n$ and $\chi\in\ggg^*$ be of standard Levi form
with $I=\{\alpha\in\Pi\mid
\chi(\ggg_{-\alpha})\neq 0\}$.

(i) Assume that
$I=\{\alpha_s,\cdots,\alpha_{n-2},\alpha_{n-1},\alpha_{n}\}, 1\leq
s\leq n-2$, i.e., we have the following Dynkin diagram
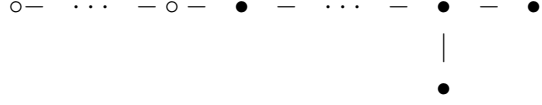
\[\begin{matrix}\circ-&\cdot\cdot\cdot&-\circ-&\bullet&-&\cdot\cdot\cdot&
-&\bullet&-&\bullet\\&&&&&&&|&&\\&&&&&&&\bullet&&\end{matrix}\]

Fix an order of the Chevalley basis
in $\uuu_J^-$ as follows
\begin{align*}
&y_{\alpha_n},y_{\alpha_{n-1}}; \\
&y_{\alpha_{n}+\alpha_{n-2}},y_{\alpha_{n-1}+\alpha_{n-2}},
y_{\alpha_{n}+\alpha_{n-1}+\alpha_{n-2}},y_{\alpha_{n-2}};\\
&\cdots\\
&y_{\alpha_{s}+\alpha_{s+1}},\cdot\cdot\cdot, y_{\alpha_{s}+\cdots+\alpha_{n-2}+\alpha_{n-1}+\alpha_{n}},
\cdot\cdot\cdot,
y_{\alpha_{s}+2\alpha_{s+1}
+2\alpha_{s+2}+\cdots+2\alpha_{n-2}+\alpha_{n-1}+\alpha_{n}},y_{\alpha_{s}};\\
&y_{\alpha_{s-1}+\alpha_{s}},\cdot\cdot\cdot, y_{\alpha_{s-1}+\cdots+\alpha_{n-2}+\alpha_{n-1}+\alpha_{n}},
\cdot\cdot\cdot, y_{\alpha_{s-1}+2\alpha_{s}
+2\alpha_{s+1}+\cdots+2\alpha_{n-2}+\alpha_{n-1}+\alpha_{n}};\\
&\cdots\\
&y_{\alpha_{1}+\alpha_{2}+\cdots+\alpha_{s}},\cdots, y_{\alpha_{1}+\cdots+\alpha_{n-2}+\alpha_{n-1}+\alpha_{n}} \cdots,y_{\alpha_{1}+2\alpha_{2}
+2\alpha_{3}+\cdots+2\alpha_{n-2}+\alpha_{n-1}+\alpha_{n}}.
\end{align*}

By a similar argument as the case of type $B_n$, we can prove Theorem \ref{main thm} in this case for type $D_n$.

Furthermore, let $\lambda=\lambda_0+p\lambda_1$ with
$\lambda_0\in X'_1(T)$ and $\lambda_1\in X(T)$. Assume that $\lambda_0$ lies in the
alcoves which contain the weight
$(1,\cdots,1,m_{s},\cdots,m_{n})-\rho$, $0<m_i<p, \forall~ i$.
Then the parabolic baby Verma module $\widehat \cz_{P}(\lambda)$ is
irreducible by Lemma \ref{T1}.

(ii) Assume that $I=\{\alpha_1,\cdots,\alpha_{n-2},\alpha_{n-1}\}$,
i.e., we have the following Dynkin diagram
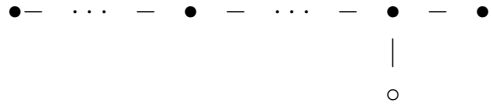
\[\begin{matrix}\bullet-&\cdot\cdot\cdot&-&\bullet&-&\cdot\cdot\cdot&-&\bullet&-&\bullet\\&&&&&&&|&&\\&&&&&&&\circ&&\end{matrix}\]

Fix an order of the Chevalley basis
in $\uuu_J^-$ as follows
\begin{align*}
&y_{\alpha_1}; \\
&y_{\alpha_{1}+\alpha_{2}},y_{\alpha_{2}};\\
&\cdots\\
&y_{\alpha_{n-2}+\alpha_{n-1}},y_{\alpha_{n-3}+\alpha_{n-2}+\alpha_{n-1}},\cdot\cdot\cdot,y_{\alpha_{1}+\alpha_{2}+\cdots+\alpha_{n-1}},
y_{\alpha_{n-1}};\\
&y_{\alpha_{n-2}+\alpha_{n}},
y_{\alpha_{n-2}+\alpha_{n-1}+\alpha_{n}},\cdot\cdot\cdot,
y_{\alpha_{1}+\cdots+\alpha_{n-2}+\alpha_{n}},y_{\alpha_{1}+\alpha_{2}+\cdots+\alpha_{n-1}+\alpha_{n}},\\ &y_{\alpha_{1}+\cdots+\alpha_{n-3}+2\alpha_{n-2}+\alpha_{n-1}+\alpha_{n}},
\cdot\cdot\cdot,
y_{\alpha_{1}+2\alpha_{2}+\cdots+2\alpha_{n-2}+\alpha_{n-1}+\alpha_{n}}.
\end{align*}

By a similar argument as the case of type $B_n$, we
can prove Theorem \ref{main thm} in this case.

Furthermore, let $\lambda=\lambda_0+p\lambda_1$ with $\lambda_0\in X'_1(T)$ and
$\lambda_1\in X(T)$, and $\lambda_0$ lies in the alcoves which
contain the weight $(m_{1},\cdots,m_{n-2},m_{n-1},1)-\rho, 0<m_i<p,
\forall~ i$, then
$\widehat \cz_{P}(\lambda)$ is irreducible by Lemma \ref{T1}.

\section{Proof of Corollary \ref{cor1}}
\subsection{}Assume that $\ggg$ is of type $A_n$ and $\chi\in\ggg^*$ has standard Levi form
associated with $I=\{\alpha\in\Pi\mid
\chi(\ggg_{-\alpha})\neq 0\}=\{\alpha_1,\alpha_2,\cdots,\alpha_{n-1}\}$.  Set
$\sigma=s_1s_2\cdots s_n$ where $s_i$ is the simple reflection
corresponding to $\alpha_i$. Assume $\lambda_0\in C_0$ with
$\lambda_0+\rho=(r_1,r_2,\cdots,r_n)$. Then
$0\leq\sum\limits_{i=1}^nr_i\leq p$.

Set $\lambda_i=\sigma^i.\lambda_0$ for $1\leq i\leq n$. Then
$\lambda_i+\rho=(r_{n-i+2},r_{n-i+3},\cdots,r_n,-(r_1+\cdots+r_n),r_1,r_2,\cdots,r_{n-i})$.
Each $\widehat L_\chi(w.\lambda_0)$ with $w\in W$ is isomorphic to
some $\widehat L_\chi(\lambda_i)$ with $0\leq i\leq n$ (cf.
\cite[\S\,2.3]{Jan4}). Then $\{\widehat L_\chi(\lambda_i)\mid 0\leq
i\leq n\}$ is the set of isomorphism classes of simple modules in
the block containing $\widehat L_\chi(\lambda_0)$.

For the decomposition $\lambda_i=\lambda_{i,0}+p\lambda_{i,1}$ with
$\lambda_{i,0}\in X'_1(T)$ and $\lambda_{i,1}\in X(T)$, since $0\leq\sum\limits_{j=1}^nr_j\leq p$, we have $\lambda_{i,0}\in
C_0,~\forall~ i$. Therefore  $\widehat \cz_{P}(\lambda_i)$ is irreducible by Theorem \ref{main thm}. Then $\widehat \cz_{P}(\lambda_i)$ has
dimension $r_{n-i}p^{N-1}$, i.e., we get part (1) of Corollary \ref{cor1} which
coincides with \cite[Theorem 2.6]{Jan4}.

\subsection{}Assume that $\ggg$ is of type $B_n$ and $\chi\in\ggg^*$ has standard Levi form
associated with $I=\{\alpha\in\Pi\mid
\chi(\ggg_{-\alpha})\neq 0\}=\{\alpha_2,\alpha_3,\cdots,\alpha_{n}\}$ (where
$\alpha_n$ is a short root). Assume $\lambda_1\in C_0$ with
$\lambda_1+\rho=(r_1,r_2,\cdots,r_n)$. Then we have
$0\leq\sum\limits_{i=1}^{n-1}2r_i+r_n\leq p$. Let

\begin{equation*}
 w_i=\begin{cases}
      s_1s_2\cdots s_{i-1}, & \text{for $1\leq i\leq n $,}\\
      s_1s_2\cdots s_{n}s_{n-1}\cdots s_{2n+1-i},  & \text{ for $n+1\leq i\leq 2n $}.
    \end{cases}
\end{equation*}
Set $\lambda_i=w_i.\lambda_1$ for $1\leq i\leq 2n$. Then $\{\widehat
L_\chi(\lambda_i)|1\leq i\leq 2n\}$ is the set of isomorphism
classes of simple modules in the block containing $\widehat
L_\chi(\lambda_1)$ (cf. \cite[\S\,3.8]{Jan4}).

Moreover, we have
$$\lambda_2+\rho=(-r_{1},r_{1}+r_{2},r_3,\cdots,r_{2n}),$$
$$\lambda_3+\rho=(-(r_{1}+r_{2}),r_{1},r_{2}+r_3,\cdots,r_{2n}),$$
$$\cdots\cdots\cdots\cdots\cdots\cdots,$$
$$\lambda_{2n}+\rho=(-(r_{1}+2r_{2}+2r_{3}+\cdots+2r_{n-1}+r_{n}),r_{2},r_3,\cdots,r_{2n}).$$
Let

\begin{equation*}
 \lambda_i'=\begin{cases}
      s_2s_3\cdots s_{i}.\lambda_i, & \text{for $2\leq i\leq n-1 $,}\\
      s_2s_3\cdots s_{n}s_{n-1}\cdots s_{2n+1-i}.\lambda_i,  & \text{ for $n+1\leq i\leq 2n-1 $}.
    \end{cases}
\end{equation*}

Then
$$\lambda_2'+\rho=(r_2,-(r_{1}+r_{2}),r_{1}+r_{2}+r_{3},\cdots,r_n),$$
$$\lambda_3'+\rho=(r_3,-(r_{1}+r_{2}+r_{3}),r_{1},r_{2}+r_{3}+r_{4},r_{5},\cdots,r_n),$$
$$\cdots\cdots\cdots\cdots\cdots\cdots,$$
$$\lambda_{2n-1}'+\rho=(r_1,-(r_{1}+r_{2}+2r_{3}+\cdots+2r_{n-1}+r_{n}),r_{3},\cdots,r_n).$$
It is obvious that the first component of $\lambda_i'+\rho$ is $r_i$
for $1\leq i\leq n-1$ and $r_{2n-i}$ for $n+1\leq i\leq 2n-1 $. By
\cite[Proposition 11.9]{Jan1}, we have $\widehat
L_\chi(\lambda_i)\cong \widehat L_\chi(\lambda_i')$ for $i\neq
n,2n$.

Since $0\leq r_i \leq p$ and
$0\leq\sum\limits_{i=1}^{n-1}2r_i+r_n\leq p$, it follows from Remark
\ref{reamrk B} that $\widehat \cz_{P}(\lambda'_i)$ ($i\neq n,\,2n$) is irreducible with
dimension $r_{i}p^{N-1}$ for $1\leq i\leq n-1$ and $r_{2n-i}p^{N-1}$
for $n+1\leq i\leq 2n-1 $, i.e., we get part (2) of Corollary \ref{cor1} which coincides with \cite[Proposition 3.13]{Jan4}.

\end{document}